\documentclass[a4, 12pt]{amsart}
\usepackage{amssymb}
\usepackage{amstext}
\usepackage{amsmath}
\usepackage{amscd}
\usepackage{amsthm}
\usepackage{amsfonts}
\usepackage{enumerate}
\usepackage{graphicx}
\usepackage{latexsym}
\theoremstyle{plain}
\newtheorem{thm}{Theorem}[section]
\newtheorem{prop}[thm]{Proposition}
\newtheorem{lem}[thm]{Lemma}
\newtheorem{cor}[thm]{Corollary}
\newtheorem*{thm*}{Theorem}
\theoremstyle{definition}
\newtheorem{dfn}[thm]{Definition}
\newtheorem{ex}[thm]{Example}
\newtheorem{ques}[thm]{Question}
\newtheorem{rem}[thm]{Remark}
\newtheorem{conv}[thm]{Convention}

\theoremstyle{remark}
\newtheorem*{ac}{Acknowledgments}
\numberwithin{equation}{thm}

\def\Hom{\operatorname{Hom}}
\def\depth{\operatorname{depth}}

\def\tr{\operatorname{Tr}}
\def\syz{\mathrm{\Omega}}
\def\Ext{\operatorname{Ext}}

\def\id{\mathrm{id}}
\def\Coker{\operatorname{Coker}}

\def\RHom{\operatorname{{\bf R}Hom}}

\def\edim{\operatorname{edim}}
\def\e{\operatorname{e}}
\def\mod{\operatorname{\mathsf{mod}}}
\def\CM{\operatorname{\mathsf{CM}}}
\def\add{\operatorname{\mathsf{add}}}
\def\res{\operatorname{\mathsf{res}}}
\def\Spec{\operatorname{\mathsf{Spec}}}
\def\Sing{\operatorname{\mathsf{Sing}}}
\def\Ass{\operatorname{\mathsf{Ass}}}
\def\Min{\operatorname{\mathsf{Min}}}
\def\ng{\operatorname{\mathsf{NonGor}}}

\def\s{\operatorname{\mathsf{S}}}
\def\m{\mathfrak m}
\def\p{\mathfrak p}
\def\q{\mathfrak q}

\def\X{{\mathcal X}}
\def\Y{{\mathcal Y}}
\def\c{{\mathcal C}}
\def\V{{\mathbb V}}
\def\C{{\mathbb C}}

\def\lp{{}_{\CM}^{\hspace{6pt}\perp}}
\def\rp{_{\hspace{2pt}\CM}^\perp}
\begin{document}
\title[Classifying resolving subcategories]{Classifying resolving subcategories over a Cohen-Macaulay local ring}
\author{Ryo Takahashi}
\address{Department of Mathematical Sciences, Faculty of Science, Shinshu University, 3-1-1 Asahi, Matsumoto, Nagano 390-8621, Japan/Department of Mathematics, University of Nebraska, Lincoln, NE 68588-0130, USA}
\email{takahasi@math.shinshu-u.ac.jp}
\urladdr{http://math.shinshu-u.ac.jp/~takahasi/}
\thanks{The author was partially supported by JSPS Grant-in-Aid for Young Scientists (B) 22740008 and by JSPS Postdoctoral Fellowships for Research Abroad}
\begin{abstract}
Let $R$ be a Cohen-Macaulay local ring.
Denote by $\mod R$ the category of finitely generated $R$-modules.
In this paper, we consider the classification problem of resolving subcategories of $\mod R$ in terms of specialization-closed subsets of $\Spec R$.
We give a classification of the resolving subcategories closed under tensor products and transposes.
Under restrictive hypotheses, we also give better classification results.
\end{abstract}
\maketitle

\section{Introduction}\label{sec1}

In the present paper, we are interested in classifying resolving subcategories of the category of finitely generated modules over a Cohen-Macaulay local ring.
One of the main results of this paper is the following, which will be proved in Theorem \ref{tentr}.

\begin{thm}
Let $R$ be a Cohen-Macaulay local ring.
Taking the nonfree loci makes a bijection between:
\begin{itemize}
\item
the resolving subcategories of $\mod R$ closed under tensor products and transposes,
\item
the specialization-closed subsets of $\Spec R$ contained in $\s(R)$.
\end{itemize}
\end{thm}

\noindent
Here, $\mod R$ denotes the category of finitely generated $R$-module and $\s(R)$ denotes the set of prime ideals $\p$ such that $R_\p$ is not a field.

Under more restrictive hypotheses, we will give better results.
More precisely, setting certain local conditions on the punctured spectrum, we will consider removing closure under tensor products and transposes from the above theorem; see Theorem \ref{mmthm} and Corollary \ref{fcrtcor}.
These two results should also be compared with \cite[Theorem 5.13]{stcm}, which gives a similar classification under a different local condition on the punctured spectrum.

We will also study in Section \ref{sec4} classifying thick subcategories of $\mod R$ containing $R$, and consider in Section \ref{sec6} determining the contravariantly finite resolving subcategories of Cohen-Macaulay modules containing a canonical module.

\section{Preliminaries}\label{sec2}

In this section, we make several definitions of the notions which we will deal with in this paper, and state their basic properties for later use.
We begin with setting our convention.

\begin{conv}
\begin{enumerate}[(1)]
\item
Throughout the rest of this paper, we assume that all rings are commutative Noetherian rings, and all modules are finitely generated.
Let $R$ be a local ring of (Krull) dimension $d$.
We denote by $\m$ the maximal ideal of $R$, by $k$ the residue field of $R$ and by $\mod R$ the category of finitely generated $R$-modules.
\item
Let $\c$ be a category.
In this paper, a {\em subcategory} of $\c$ always mean a strict full subcategory of $\c$.
(Recall that a subcategory $\X$ of $\c$ is said to be {\em strict} if every object of $\c$ that is isomorphic in $\c$ to some object of $\X$ belongs to $\X$.)
The {\em subcategory} of $\c$ consisting of objects $\{M_\lambda\}_{\lambda\in\Lambda}$ always mean the smallest strict full subcategory of $\c$ to which $M_\lambda$ belongs for all $\lambda\in\Lambda$.
Note that this coincides with the full subcategory of $\c$ consisting of all objects $X\in\c$ such that $X\cong M_\lambda$ for some $\lambda\in\Lambda$.
\item
For $n\ge0$, the {\em $n$-th syzygy} of an $R$-module $M$ is defined to be the image of the $n$-th differential map in a minimal free resolution of $M$, and is denoted by $\syz_R^nM$.
We simply write $\syz_RM$ instead of $\syz_R^1M$.
Note that the $n$-th syzygy of a given $R$-module is uniquely determined up to isomorphism because so is a minimal free resolution.
\item
When $R$ is a Cohen-Macaulay local ring, we say that an $R$-module $M$ is {\em Cohen-Macaulay} if $\depth_RM=d$.
Such a module is usually called a {\em maximal} Cohen-Macaulay $R$-module, but in this paper, we call it just Cohen-Macaulay.
We denote by $\CM(R)$ the category of Cohen-Macaulay $R$-modules.
Note that a subcategory of $\CM(R)$ can naturally be regarded as a subcategory of $\mod R$.
\item
We will often omit a letter indicating the base ring if there is no danger of confusion.
For example, we will often write $\Ext^i(M,N)$, $\syz^iM$ and $M\otimes N$ instead of $\Ext_R^i(M,N)$, $\syz_R^iM$ and $M\otimes_RN$, respectively.
\end{enumerate}
\end{conv}

\begin{dfn}
\begin{enumerate}[(1)]
\item
A subcategory of $\mod R$ is called {\em resolving} if it contains $R$, and if it is closed under direct summands, extensions and kernels of epimorphisms.
\item
Let $R$ be Cohen-Macaulay.
A subcategory $\X$ of $\CM(R)$ is called {\em thick} provided that $\X$ is closed under direct summands and that for each exact sequence $0 \to L \to M \to N \to 0$ of Cohen-Macaulay $R$-modules, if two of $L,M,N$ are in $\X$, then so is the third.
\end{enumerate}
\end{dfn}

A subcategory $\X$ of $\mod R$ is resolving if and only if $\X$ contains $R$ and is closed under direct summands, extensions and syzygies (cf. \cite[Lemma 3.2]{Y2}).
A lot of subcategories of $\mod R$ are known to be resolving; see \cite[Example 1.6]{stcm}.
Every thick subcategory of $\CM(R)$ containing $R$ is a resolving subcategory of $\mod R$.
One can construct thick subcategories of $\CM(R)$ by restricting resolving subcategories appearing in \cite[Example 1.6]{stcm} to $\CM(R)$.

\begin{dfn}
The {\em nonfree locus} $\V_R(M)$ of an $R$-module $M$ is defined as the set of prime ideals $\p$ of $R$ such that $M_\p$ is nonfree as an $R_\p$-module.
For a subcategory $\X$ of $\mod R$ we set $\V_R(\X)=\bigcup_{M\in\X}\V_R(M)$ and call it the {\em nonfree locus} of $\X$.
\end{dfn}

Here $\V_R(M),\V_R(\X)$ are the same as $\mathcal{V}_R(M),\mathcal{V}_R(\X)$ in \cite{stcm}.
As is well-known, the nonfree locus of an $R$-module is always a closed subset of $\Spec R$ in the Zariski topology.

Let $\X$ be a subcategory of $\mod R$.
We denote by $\add_R\X$ the {\em additive closure} of $\X$, that is, the subcategory of $\mod R$ consisting of all direct summands of finite direct sums of modules in $\X$.
For a prime ideal $\p$ of $R$, we denote by $\X_\p$ the subcategory of $\mod R_\p$ consisting of all $R_\p$-modules of the form $X_\p$ with $X\in\X$.

\begin{prop}\label{lnx}
Let $\X$ be a resolving subcategory of $\mod R$, and let $M$ be an $R$-module.
Let $\Gamma\ne\emptyset$ be a finite subset of $\Spec R$.
Assume $M_\p\in\add_{R_\p}\X_\p$ for every $\p\in\Gamma$.
Then
\begin{enumerate}[\rm (1)]
\item
There exists an exact sequence $0 \to L \to N \to X \to 0$ of $R$-modules satisfying the following four conditions:
\begin{enumerate}[\rm (i)]
\item
$X$ belongs to $\X$,
\item
$M$ is a direct summand of $N$,
\item
$\V_R(L)$ is contained in $\V_R(M)$,
\item
$\V_R(L)$ does not intersect $\Gamma$.
\end{enumerate}
\item
Suppose that $R$ is a Cohen-Macaulay local ring, that $\X$ is contained in $\CM(R)$ and that $M$ is a Cohen-Macaulay $R$-module.
Then the modules $L,N$ in the exact sequence in (1) can be chosen as Cohen-Macaulay $R$-modules.
\end{enumerate}
\end{prop}

\begin{proof}
The second statement is stated in \cite[Proposition 4.7]{stcm}.
Along the same lines as in its proof, one can show the first statement.
\end{proof}

Let $\Phi$ be a subset of $\Spec R$.
We define the {\em dimension} $\dim\Phi$ of $\Phi$ as the supremum of $\dim R/\p$ with $\p\in\Phi$.
By definition, one has $\dim\Phi=-\infty$ if and only if $\Phi$ is empty.
We denote by $\min\Phi$ the set of minimal elements of $\Phi$ with respect to inclusion relation.
The subcategory of $\mod R$ consisting of all $R$-modules whose nonfree loci are contained in $\Phi$ is denoted by $\V^{-1}(\Phi)$.
When $R$ is Cohen-Macaulay, $\V^{-1}_{\CM}(\Phi)$ denotes the restriction of $\V^{-1}(\Phi)$ to $\CM(R)$.
We denote by $\s(R)$ the set of prime ideals $\p$ such that $R_\p$ is not a field.
The {\em singular locus} $\Sing R$ is defined as the set of prime ideals $\p$ such that $R_\p$ is singular, i.e., nonregular.
Recall that $\Phi$ is called {\em specialization-closed} if every prime ideal of $R$ containing some prime ideal in $\Phi$ belongs to $\Phi$.
The following two lemmas will frequently be used as basic tools in the proofs of our results.
The second one is a Cohen-Macaulay module version of the first one; Lemma \ref{cv}($i$) corresponds to Lemma \ref{nv}($i$) for each $1\le i\le 10$.

\begin{lem}\label{nv}
Let $R$ be a Cohen-Macaulay local ring, $M$ an $R$-module, $\X$ a resolving subcategory of $\mod R$ and $\Phi$ a specialization-closed subset of $\Spec R$ contained in $\s(R)$.
\begin{enumerate}[\rm (1)]
\item
If $\dim\V(M)=-\infty$, then $M\in\X$.
\item
If $\dim\V(M)=0$ and $k\in\X$, then $M\in\X$.
\item
If $\p\in\min\V_R(M)$, then $\V_{R_\p}(M_\p)=\{\p R_\p\}$, and $\dim\V_{R_\p}(M_\p)=0$.
\item
$\add_{R_\p}\X_\p$ is a resolving subcategory of $\mod R_\p$.
\item
If $M_\p\in\add_{R_\p}\X_\p$ for all $\p\in\min\V(M)$, then there is an exact sequence $0 \to L \to N \to X \to 0$ of $R$-modules with $X\in\X$, $\V(L)\subseteq\V(M)$ and $\dim\V(L)<\dim\V(M)$ such that $M$ is a direct summand of $N$.
\item
$\V(\X)$ is a specialization-closed subset of $\Spec R$ contained in $\s(R)$.
\item
$R/\p$ belongs to $\V^{-1}(\Phi)$ for all $\p\in\Phi$.
\item
$\V^{-1}(\Phi)$ is a resolving subcategory of $\mod R$.
\item
$\X$ is contained in $\V^{-1}(\V(\X))$.
\item
One has $\Phi=\V(\V^{-1}(\Phi))$.
\end{enumerate}
\end{lem}

\begin{proof}
(2) If $\dim\V(M)=0$, then we have $\V(M)=\{\m\}$.
This implies that $M$ is locally free on the punctured spectrum of $R$.
According to \cite[Theorem 2.4]{stcm}, the $R$-module $M$ can be constructed from $k$ by taking syzygies, extensions and direct summands finitely many times.
Since $\X$ contains $k$ and is closed under these operations, $M$ belongs to $\X$.

(4)(7)(8) See \cite[Lemma 4.8 and Proposition 1.15]{stcm}.

(5) Apply Proposition \ref{lnx}(1) to the set $\Gamma:=\min\V(M)$.

(10) If $\p\in\Phi$, then $\p\in\V_R(R/\p)$, and $\p\in\V(\V^{-1}(\Phi))$ by (7).
\end{proof}

\begin{lem}\label{cv}
Let $R$ be a Cohen-Macaulay local ring, $M$ a Cohen-Macaulay $R$-module, $\X$ a resolving subcategory of $\mod R$ contained in $\CM(R)$ and $\Phi$ a specialization-closed subset of $\Spec R$ contained in $\Sing R$.
\begin{enumerate}[\rm (1)]
\item
If $\dim\V(M)=-\infty$, then $M\in\X$.
\item
If $\dim\V(M)=0$ and $\syz^dk\in\X$, then $M\in\X$.
\item
If $\p\in\min\V_R(M)$, then $\V_{R_\p}(M_\p)=\{\p R_\p\}$, and $\dim\V_{R_\p}(M_\p)=0$.
\item
$\add_{R_\p}\X_\p$ is a resolving subcategory of $\mod R_\p$ contained in $\CM(R_\p)$.
\item
If $M_\p\in\add_{R_\p}\X_\p$ for all $\p\in\min\V(M)$, then there is an exact sequence $0 \to L \to N \to X \to 0$ of Cohen-Macaulay $R$-modules with $X\in\X$, $\V(L)\subseteq\V(M)$ and $\dim\V(L)<\dim\V(M)$ such that $M$ is a direct summand of $N$.
\item
$\V(\X)$ is a specialization-closed subset of $\Spec R$ contained in $\Sing R$.
\item
$\syz^d(R/\p)$ belongs to $\V^{-1}_{\CM}(\Phi)$ for all $\p\in\Phi$.
\item
$\V^{-1}_{\CM}(\Phi)$ is a resolving subcategory of $\mod R$ contained in $\CM(R)$.
\item
$\X$ is contained in $\V^{-1}_{\CM}(\V(\X))$.
\item
One has $\Phi=\V(\V^{-1}_{\CM}(\Phi))$.
\end{enumerate}
\end{lem}

\begin{proof}
(2)(5) In the proof of the corresponding statement in Lemma \ref{nv}, use \cite[Corollary 2.6]{stcm}/Proposition \ref{lnx}(2) instead of \cite[Theorem 2.4]{stcm}/Proposition \ref{lnx}(1).

(7) This follows from Lemma \ref{nv}(7)(8).

(8) Lemma \ref{nv}(8) and \cite[Example 1.6(3)]{stcm} imply this assertion.
\end{proof}

\section{A local condition}\label{sec3}

In this section, we give a classification theorem of resolving subcategories satisfying a certain local condition.
For a prime ideal $\p$ of $R$, we denote the residue field $R_\p/\p R_\p$ by $\kappa(\p)$.
The proposition below forms the essential part of the classification theorem.

\begin{prop}\label{resmodkey}
Let $R$ be a Cohen-Macaulay local ring.
Let $\X$ be a resolving subcategory of $\mod R$ and $M$ an $R$-module such that $\V_R(M)\subseteq\V_R(\X)$.
If $\kappa(\p)\in\add_{R_\p}\X_\p$ for each $\p\in\V_R(M)$, then $M$ is in $\X$.
\end{prop}

\begin{proof}
The assertion is proved by induction on $\dim\V_R(M)$.
When $\dim\V_R(M)=-\infty$, Lemma \ref{nv}(1) shows $M\in\X$.
When $\dim\V_R(M)=0$, we have $\V_R(M)=\{\m\}$.
By assumption, $\kappa(\m)\in\add_{R_\m}\X_\m$, which means $k\in\X$.
Lemma \ref{nv}(2) implies $M\in\X$.

Let us consider the case where $\dim\V_R(M)\ge 1$.
Take $\p\in\min\V_R(M)$.
We have $\dim\V_{R_\p}(M_\p)=0<\dim\V_R(M)$ by Lemma \ref{nv}(3), and $\add_{R_\p}\X_\p$ is a resolving subcategory of $\mod R_\p$ by Lemma \ref{nv}(4).
We see that $\kappa(\q)\in\add_{(R_\p)_\q}(\add_{R_\p}\X_\p)_\q$ for every $\q\in\V_{R_\p}(M_\p)$.
We have $\V_{R_\p}(M_\p)=\{\p R_\p\}\subseteq\V_{R_\p}(\add_{R_\p}\X_\p)$ by Lemma \ref{nv}(3).
The basis of the induction shows $M_\p\in\add_{R_\p}\X_\p$ for all $\p\in\min\V_R(M)$.
By Lemma \ref{nv}(5) there exists an exact sequence $0 \to L \to N \to X \to 0$ with $X\in\X$, $\V_R(L)\subseteq\V_R(M)$ and $\dim\V_R(L)<\dim\V_R(M)$ such that $M$ is a direct summand of $N$.
We have $\kappa(\p)\in\add_{R_\p}\X_\p$ for all $\p\in\V_R(L)$, and $\V_R(L)\subseteq\V_R(\X)$ holds.
The induction hypothesis shows $L\in\X$.
From the above exact sequence we see that $N\in\X$, and $M\in\X$.
\end{proof}

Proposition \ref{resmodkey} does not necessarily hold if the condition that $\kappa(\p)$ belongs to $\add_{R_\p}\X_\p$ for all $\p\in\V_R(M)$ is replaced with the weaker condition that $k$ belongs to $\X$.

\begin{ex}
Let $R=k[[x,y,z]]$ be a formal power series ring over a field $k$.
Then the maximal ideal of $R$ is $\m=(x,y,z)$, and $\p=(x,y)$, $\q=(x)$ are prime ideals of $R$.
Let $\X$ be the smallest resolving subcategory of $\mod R$ containing $k$ and $R/\q$.
Then $k$ belongs to $\X$ and $\V_R(R/\p)$ is contained in $\V_R(\X)$, but $R/\p$ does not belong to $\X$.

Indeed, we have $\V_R(R/\p)=\{\p,\m\}$.
Since $\p,\m$ are in $\V_R(R/\q)$, they belong to $\V_R(\X)$.
Hence $\V_R(R/\p)$ is contained in $\V_R(\X)$.
By \cite[Proposition 1.10(2)]{stcm}, every module $X\in\X$ satisfies $\Ass_RX\subseteq\Ass_R(k\oplus R/\q)\cup\Ass_RR=\{\m,\q,0\}$.
(Here the set of associated prime ideals of an $R$-module $M$ is denoted by $\Ass_RM$.)
It follows from this that $R/\p\notin\X$.
\end{ex}

Now, let us prove the main result of this section.
Note that in the one-to-one correspondence the smallest resolving subcategory $\add R$ corresponds to the empty set.

\begin{thm}\label{resmod}
Let $R$ be a Cohen-Macaulay local ring.
Then one has the following one-to-one correspondence:
$$
\begin{CD}
\left\{
\begin{matrix}
\text{Resolving subcategories $\X$ of $\mod R$}\\
\text{with $\kappa(\p)\in\add_{R_\p}\X_\p$ for all $\p\in\V(\X)$}
\end{matrix}
\right\}
\begin{matrix}
@>{\V}>>\\
@<<{\V^{-1}}<
\end{matrix}
\left\{
\begin{matrix}
\text{Specialization-closed subsets}\\
\text{of $\Spec R$ contained in $\s(R)$}
\end{matrix}
\right\}.
\end{CD}
$$
\end{thm}

\begin{proof}
If $\Phi$ is a specialization-closed subset of $\Spec R$ contained in $\s(R)$, then $\kappa(\p)=(R/\p)_\p$ belongs to $\add_{R_\p}(\V^{-1}(\Phi))_\p$ for every $\p\in\Phi=\V(\V^{-1}(\Phi))$ by Lemma \ref{nv}(7)(10).
Hence, it follows from Lemma \ref{nv}(6)(8) that the maps $\V$ and $\V^{-1}$ are well-defined.
By Lemma \ref{nv}(9)(10) and Proposition \ref{resmodkey}, the two maps are mutually inverse bijections.
\end{proof}

\section{Tensor products and transposes}\label{sec4}

In this section we provide an alternative characterization of the resolving subcategories arising in Theorem \ref{resmod}.
We start by introducing the transpose of a module with respect to a fixed module.

\begin{dfn}\label{deftrx}
Let $M$ be an $R$-module.
\begin{enumerate}[(1)]
\item
Let $\sigma: P_1 \to P_0 \to M \to 0$ be an exact sequence of $R$-modules such that $P_0$, $P_1$ are free.
Then we call $\sigma$ a {\em free presentation} of $M$.
If $\sigma$ can be extended to a minimal free resolution of $M$, then we call $\sigma$ a {\em minimal free presentation} of $M$.
\item
Let $X$ be an $R$-module.
Take a minimal free presentation $F_1 \overset{\partial}{\to} F_0 \to M \to 0$ of $M$.
Then we define the {\em transpose} of $M$ {\em with respect to} $X$ as the cokernel of the dual map $\Hom_R(\partial,X):\Hom_R(F_0,X)\to\Hom_R(F_1,X)$, and denote it by $\tr_XM$.
\end{enumerate}
\end{dfn}

\begin{rem}
With the notation of Definition \ref{deftrx}, one has:
\begin{enumerate}[(1)]
\item
The module $\tr_XM$ is uniquely determined up to isomorphism since so is a minimal free resolution of $M$.
\item
There is an exact sequence
$$
\begin{CD}
0 \to \Hom_R(M,X) \to \Hom_R(F_0,X) @>{\Hom_R(\partial,X)}>> \Hom_R(F_1,X) \to \tr_XM \to 0.
\end{CD}
$$
\item
The transpose $\tr_RM$ with respect to $R$ is nothing but the usual {\em (Auslander) transpose} of $M$.
Recall that $M$ is isomorphic to $\tr_R(\tr_RM)$ up to free summand.
\end{enumerate}
\end{rem}

Here are some basic properties of transposes with respect to a fixed module, which will be necessary later.

\begin{prop}\label{trx}
Let $X$ be an $R$-module.
\begin{enumerate}[\rm (1)]
\item
For an $R$-module $M$ one has an isomorphism $\tr_XM\cong X\otimes_R\tr_RM$.
\item
Let $P_1\overset{\partial}{\to} P_0 \to M \to 0$ be a free presentation of an $R$-module $M$.
Then there is an exact sequence
$$
\begin{CD}
0 \to \Hom_R(M,X) \to \Hom_R(P_0,X) @>{\Hom_R(\partial,X)}>> \Hom_R(P_1,X) \to \tr_XM\oplus X^{\oplus n} \to 0.
\end{CD}
$$
\item
Let $0 \to L \overset{f}{\to} M \overset{g}{\to} N \to 0$ be an exact sequence of $R$-modules.
Then there is an exact sequence
$$
\begin{CD}
0 @>>> \Hom_R(N,X) @>{\Hom_R(g,X)}>> \Hom_R(M,X) @>{\Hom_R(f,X)}>> \Hom_R(L,X) \\
@>>> \tr_XN @>>> \tr_XM\oplus X^{\oplus n} @>>> \tr_XL @>>> 0.
\end{CD}
$$
\end{enumerate}
\end{prop}

\begin{proof}
(1) Let $F_1 \to F_0 \to M \to 0$ be a minimal free presentation of $M$.
We have a commutative diagram
$$
\begin{CD}
X\otimes\Hom(F_0,R) @>>> X\otimes\Hom(F_1,R) @>>> X\otimes\tr M @>>> 0 \\
@V{f_0}VV @V{f_1}VV @V{g}VV \\
\Hom(F_0,X) @>>> \Hom(F_1,X) @>>> \tr_XM @>>> 0
\end{CD}
$$
with exact rows, where $f_i(x\otimes h)(y)=h(y)x$ for $x\in X$, $h\in\Hom(F_i,R)$ and $y\in F_i$.
Since $f_0$ and $f_1$ are isomorphisms, so is $g$.

(2) Extend the free presentation of $M$ to a free resolution of $M$ and use uniqueness of minimal free resolutions (cf. \cite[Theorem 20.2]{E}).
We can easily obtain an isomorphism $\Coker\Hom(\partial,X)\cong\tr_XM\oplus X^{\oplus n}$.

(3) First, take minimal free presentations of $L$ and $N$.
Second, apply the horseshoe lemma to them to make a commutative diagram.
Third, apply $\Hom(-,X)$ to the diagram.
Fourth, use (2) and the snake lemma.
Then we obtain a desired exact sequence.
\end{proof}

Let $\X$ be a subcategory of $\mod R$.
We say that $\X$ is {\em closed under tensor products} if for all $X,Y\in\X$ the tensor product $X\otimes_RY$ belongs to $\X$.
We say that $\X$ is {\em closed under transposes} if the transpose $\tr_RX$ is in $\X$ for every $X\in\X$.

We give two lemmas.

\begin{lem}\label{xkx}
Let $\X$ be a resolving subcategory of $\mod R$ which is closed under tensor products and transposes.
If $\X$ contains a nonfree $R$-module, then $k$ belongs to $\X$.
\end{lem}

\begin{proof}
By virtue of \cite[Theorem A]{res}, there exists a nonfree $R$-module $X\in\X$ which is locally free on the punctured spectrum of $R$.
Since $\Ext^1(X,\syz X)$ is a nonzero $R$-module of finite length by \cite[Corollary 2.11]{res}, its socle is nonzero, and we can choose an element $0\ne\sigma\in\Ext^1(X,\syz X)$ which is annihilated by the maximal ideal $\m$.
Let $0 \to \syz X \to Y \to X \to 0$ be an exact sequence of $R$-modules which corresponds to $\sigma$.
Note that $Y$ is in $\X$.
Using Proposition \ref{trx}(3), we get an exact sequence
$$
\begin{CD}
0 @>>> \Hom(X,\syz X) @>>> \Hom(Y,\syz X) @>{f}>> \Hom(\syz X,\syz X) \\
@>>> \tr_{\syz X}X @>>> \tr_{\syz X}Y\oplus(\syz X)^{\oplus n} @>>> \tr_{\syz X}\syz X @>>> 0.
\end{CD}
$$
There is an exact sequence
$$
\Hom(Y,\syz X) \overset{f}{\to} \Hom(\syz X,\syz X) \overset{g}{\to} \Ext^1(X,\syz X),
$$
where $g$ sends the identity map $\id_{\syz X}\in\Hom(\syz X,\syz X)$ to $\sigma\in\Ext^1(X,\syz X)$.
Hence the cokernel of $f$ is isomorphic to the $R$-submodule $R\sigma$ of $\Ext^1(X,\syz X)$, which is isomorphic to $k$.
Therefore we obtain an exact sequence
$$
0 \to k \to \tr_{\syz X}X \to \tr_{\syz X}Y\oplus(\syz X)^{\oplus n} \to \tr_{\syz X}\syz X \to 0.
$$
Proposition \ref{trx}(1) shows that $\tr_{\syz X}X$ is isomorphic to $\syz X\otimes\tr X$, which belongs to $\X$ since $\X$ is closed under tensor products and transposes.
Hence $\tr_{\syz X}X$ belongs to $\X$.
Similarly, $\tr_{\syz X}Y$ and $\tr_{\syz X}\syz X$ are also in $\X$.
Decomposing the above exact sequence into two short exact sequences, we observe that $k$ is in $\X$.
\end{proof}

\begin{lem}\label{imm}
Let $M$ and $N$ be $R$-modules.
Then the following hold.
\begin{enumerate}[\rm (1)]
\item
One has $\V_R(M\otimes_RN)\subseteq\V_R(M)\cup\V_R(N)\supseteq\V_R(\Hom_R(M,N))$.
\item
One has $\V_R(M)=\V_R(\tr_RM)$.
\end{enumerate}
\end{lem}

\begin{proof}
Let $\p$ be a prime ideal of $R$.
If $M_\p$ and $N_\p$ are $R_\p$-free, then so are $M_\p\otimes_{R_\p}N_\p$ and $\Hom_{R_\p}(M_\p,N_\p)$, which are isomorphic to $(M\otimes_RN)_\p$ and $\Hom_R(M,N)_\p$, respectively.
Also, the $R_\p$-module $M_\p$ is free if and only if so is $\tr_{R_\p}M_\p$, which is isomorphic to $(\tr_RM)_\p$ up to free summand.
The assertions of the lemma immediately follow from these.
\end{proof}

Now we can clarify the meaning of the local condition given in the previous section.

\begin{prop}\label{vmx}
Let $R$ be a Cohen-Macaulay local ring.
Let $\X$ be a resolving subcategory of $\mod R$.
The following are equivalent:
\begin{enumerate}[\rm (1)]
\item
$\X$ is closed under tensor products and transposes;
\item
$\X$ is closed under $\tr_X(-)$ for all $X\in\X$;
\item
$R/\p$ belongs to $\X$ for every $\p\in\V(\X)$;
\item
$\kappa(\p)$ belongs to $\add_{R_\p}\X_\p$ for every $\p\in\V(\X)$.
\end{enumerate}
\end{prop}

\begin{proof}
(1) $\Rightarrow$ (2):
This is easily verified by using Proposition \ref{trx}(1).

(2) $\Rightarrow$ (1):
Closure under transposes holds as $R\in\X$.
Let $M,N\in\X$.
Note that $N\cong\tr\tr N\oplus R^{\oplus n}$ for some $n\ge0$.
We have $M\otimes N\cong M\otimes(\tr\tr N\oplus R^{\oplus n})\cong \tr_M(\tr N)\oplus M^{\oplus n}$ by Proposition \ref{trx}(1).
Hence $M\otimes N$ belongs to $\X$.

(3) $\Rightarrow$ (4): Localization at $\p$ shows this implication.

(4) $\Rightarrow$ (3): Theorem \ref{resmod} implies that the set $\Phi:=\V(\X)$ is a specialization-closed subset of $\Spec R$ contained in $\s(R)$, and that the equality $\X=\V^{-1}(\Phi)$ holds.
Let $\p$ be a prime ideal in $\V(\X)=\Phi$.
Then Lemma \ref{nv}(7) implies that $R/\p$ belongs to $\V^{-1}(\Phi)=\X$.

(4) $\Rightarrow$ (1):
It follows from Theorem \ref{resmod} that we have $\X=\V^{-1}(\Phi)$ for some set $\Phi$ of prime ideals.
Lemma \ref{imm} implies that $\X$ is closed under tensor products and transposes.

(1) $\Rightarrow$ (4):
Let $\p$ be a prime ideal in $\V(\X)$.
Then $\add_{R_\p}\X_\p$ contains a nonfree $R_\p$-module.
By Lemmas \ref{nv}(4) and \ref{xkx}, it suffices to show that $\add_{R_\p}\X_\p$ is closed under tensor products and transposes as a subcategory of $\mod R_\p$.
Let $M$ and $N$ be $R_\p$-modules in $\add_{R_\p}\X_\p$.
Then there exist $R$-modules $X,Y\in\X$ such that $M$ and $N$ are isomorphic to direct summands of $X_\p$ and $Y_\p$, respectively.
Hence the $R_\p$-module $M\otimes_{R_\p}N$ is isomorphic to a direct summand of $X_\p\otimes_{R_\p}Y_\p$.
Since $X_\p\otimes_{R_\p}Y_\p\cong(X\otimes_RY)_\p$ and $X\otimes_RY\in\X$, the module $M\otimes_{R_\p}N$ belongs to $\add_{R_\p}\X_\p$.
On the other hand, $\tr_{R_\p}M$ is isomorphic to a direct summand of $\tr_{R_\p}X_\p$, which is isomorphic to $(\tr_RX)_\p$ up to free summand.
As $\tr_RX$ is in $\X$, we have that $\tr_{R_\p}M$ is in $\add_{R_\p}\X_\p$.
Consequently, $\add_{R_\p}\X_\p$ is closed under tensor products and transposes.
\end{proof}

Proposition \ref{vmx} and Theorem \ref{resmod} yield a complete classification of the resolving subcategories closed under tensor products and transposes.

\begin{thm}\label{tentr}
Let $R$ be a Cohen-Macaulay local ring.
Then one has the following one-to-one correspondence:
$$
\begin{CD}
\left\{
\begin{matrix}
\text{Resolving subcategories of $\mod R$ closed}\\
\text{under tensor products and transposes}
\end{matrix}
\right\}
\begin{matrix}
@>{\V}>>\\
@<<{\V^{-1}}<
\end{matrix}
\left\{
\begin{matrix}
\text{Specialization-closed subsets}\\
\text{of $\Spec R$ contained in $\s(R)$}
\end{matrix}
\right\}.
\end{CD}
$$
\end{thm}

\begin{ques}
Let $R$ be a Cohen-Macaulay local ring.
Let $\X$ be a resolving subcategory of $\mod R$ closed under tensor products.
Then is $\X$ closed under transposes?
\end{ques}

Let $R$ be a Cohen-Macaulay local ring.
Recall that $R$ is said to be {\em generically Gorenstein} if $R_\p$ is a Gorenstein local ring for every $\p\in\Min R$.
(Here $\Min R$ denotes the set of minimal prime ideals of $R$.)
Also, we recall that an $R$-module $M$ is said to be {\em generically free} if $M_\p$ is a free $R_\p$-module for every $\p\in\Min R$.

\begin{ex}\label{gengf}
Let $R$ be a generically Gorenstein, Cohen-Macaulay local ring.
We consider the set
$$
\Phi=\s(R)\setminus\Min R=\Spec R\setminus\Min R
$$
of prime ideals of $R$.
It is straightforward that $\Phi$ is a specialization-closed subset of $\Spec R$ contained in $\s(R)$.
We easily observe that $\V^{-1}(\Phi)$ consists of all generically free $R$-modules.
By Theorem \ref{tentr}, the generically free $R$-modules form a resolving subcategory $\X$ of $\mod R$ which is closed under tensor products and transposes.
\end{ex}

In the rest of this section, we consider classifying thick subcategories of $\CM(R)$ containing $R$, which are special resolving subcategories of $\mod R$.
We say that a subcategory $\X$ of $\CM(R)$ (containing $R$) is {\em closed under $\syz^d\Hom$} if $\syz^d\Hom(X,Y)$ belongs to $\X$ for all $X,Y\in\X$.
A similar result to Lemma \ref{xkx} holds.

\begin{lem}\label{syzd}
Let $R$ be a Cohen-Macaulay local ring of dimension $d$.
Let $\X$ be a thick subcategory of $\CM(R)$ containing $R$ and closed under $\syz^d\Hom$.
If $\X$ contains a nonfree $R$-module, then $\syz^dk$ belongs to $\X$.
\end{lem}

\begin{proof}
As in the proof of Lemma \ref{xkx}, we can choose an $R$-module $X\in\X$ and a nonzero element $\sigma\in\Ext_R^1(X,\syz X)$ with $\m\sigma=0$, and from an exact sequence $0 \to \syz X \to Y \to X \to 0$ corresponding to $\sigma$ we get an exact sequence
$$
0 \to \Hom(X,\syz X) \to \Hom(Y,\syz X) \to \Hom(\syz X,\syz X) \to k \to 0.
$$
Applying $\syz^d$ to this gives rise to an exact sequence
$$
0 \to \syz^d\Hom(X,\syz X) \to \syz^d\Hom(Y,\syz X)\oplus R^n \to \syz^d\Hom(\syz X,\syz X)\oplus R^m \to \syz^dk \to 0
$$
of Cohen-Macaulay $R$-modules.
Note that the first, second and third terms in this exact sequence are all in $\X$.
Decomposing the exact sequence into two short exact sequences and using the thickness of $\X$, we observe that the fourth term, $\syz^dk$, is also in $\X$.
\end{proof}

For a subcategory $\X$ of $\CM(R)$, we denote by $\widetilde\X$ the subcategory of $\mod R$ consisting of all modules $M$ such that there is an exact sequence
$$
0 \to X_n \to X_{n-1} \to \cdots \to X_0 \to M \to 0
$$
with each $X_i$ is in $\X$.
An analogue of Proposition \ref{vmx} holds true.

\begin{prop}\label{xrpkp}
Let $R$ be a Cohen-Macaulay local ring with a canonical module $\omega$.
Let $\X$ be a thick subcategory of $\CM(R)$ containing $R$ and $\omega$.
The following three conditions are equivalent:
\begin{enumerate}[\rm (1)]
\item
$\X$ is closed under $\syz^d\Hom$;
\item
$R/\p$ belongs to $\widetilde\X$ for every $\p\in\V(\X)$;
\item
$\kappa(\p)$ belongs to $\widetilde{\add_{R_\p}\X_\p}$ for every $\p\in\V(\X)$.
\end{enumerate}
\end{prop}

\begin{proof}
(2) $\Leftrightarrow$ (3): This equivalence follows from \cite[Corollary 4.11]{stcm}.

(1) $\Rightarrow$ (3): Let $\p$ be a prime ideal in $\V(\X)$.
Then $\add_{R_\p}\X_\p$ contains a nonfree $R_\p$-module.
The argument (3) in the proof of \cite[Proposition 4.9]{stcm} shows that $\add_{R_\p}\X_\p$ is a thick subcategory of $\CM(R_\p)$ containing $R_\p$.

Set $e=\dim R_\p$.
Let us check that $\add_{R_\p}\X_\p$ is, as a subcategory of $\mod R_\p$, closed under $\syz^e\Hom$.
Let $M,N$ be $R_\p$-modules in $\add_{R_\p}\X_\p$.
Then $M,N$ are isomorphic to direct summands of $X_\p,Y_\p$ for some $X,Y\in\X$, and $\syz^e_{R_\p}\Hom_{R_\p}(M,N)$ is isomorphic to a direct summand of $\syz^e_{R_\p}\Hom_{R_\p}(X_\p,Y_\p)$, which is isomorphic up to free summand to $(\syz^e_R\Hom_R(X,Y))_\p$.
Since $\syz^e_R\Hom_R(X,Y)$ is in $\X$, the module $\syz^e_{R_\p}\Hom_{R_\p}(M,N)$ belongs to $\add_{R_\p}\X_\p$.
Thus $\add_{R_\p}\X_\p$ is closed under $\syz^e\Hom$.

Now Lemma \ref{syzd} shows that $\syz^e\kappa(\p)$ belongs to $\add_{R_\p}\X_\p$, which implies that $\kappa(\p)$ belongs to $\widetilde{\add_{R_\p}\X_\p}$, as desired.

(3) $\Rightarrow$ (1): By \cite[Theorem 4.10]{stcm}, we have $\X=\V^{-1}_{\CM}(\Phi)$ for some subset $\Phi$ of $\Spec R$.
Let $X,Y$ be modules in $\X$.
Then by Lemma \ref{imm}(1), $\Hom(X,Y)$ belongs to $\V^{-1}(\Phi)$.
Hence $\syz^d\Hom(X,Y)$ is in $\V^{-1}_{\CM}(\Phi)=\X$ since $\V^{-1}(\Phi)$ is closed under syzygies by Lemma \ref{nv}(8), and thus $\X$ is closed under $\syz^d\Hom$.
\end{proof}

Combining \cite[Theorem 4.10]{stcm} and Proposition \ref{xrpkp}, we obtain the following classification result of thick subcategories.
Here, $\ng(R)$ denotes the {\em non-Gorenstein locus} of $R$, that is, the set of prime ideals $\p$ of $R$ such that the local ring $R_\p$ is non-Gorenstein.

\begin{thm}
Let $R$ be a Cohen-Macaulay local ring with a canonical module $\omega$.
Then one has the following one-to-one correspondence:
$$
\begin{CD}
\left\{
\begin{matrix}
\text{Thick subcategories of $\CM(R)$}\\
\text{containing $R,\omega$}\\
\text{and closed under $\syz^d\Hom$}
\end{matrix}
\right\}
\begin{matrix}
@>{\V}>>\\
@<<{\V^{-1}_{\CM}}<
\end{matrix}
\left\{
\begin{matrix}
\text{Specialization-closed subsets}\\
\text{of $\Spec R$ contained in $\Sing R$}\\
\text{and containing $\ng(R)$}
\end{matrix}
\right\}.
\end{CD}
$$
\end{thm}

\section{Minimal multiplicity}\label{sec5}

In this section, we consider the classification problem of resolving subcategories over a Cohen-Macaulay local ring with minimal multiplicity.
Recall that a Cohen-Macaulay local ring $R$ always satisfies the inequality $$\e(R)\ge \edim R-\dim R+1$$, where $\e(R)$ denotes the multiplicity of $R$ and $\edim R$ denotes the embedding dimension of $R$, and that $R$ is said to have {\em minimal multiplicity} if the equality holds.

First of all, we prove that one can determine all resolving subcategories over an Artinian local ring whose maximal ideal is square zero.

\begin{lem}\label{m^2}
If $\m^2=0$, then all resolving subcategories of $\mod R$ are $\add R$ and $\mod R$.
\end{lem}

\begin{proof}
Let $\X\ne\add R$ be a resolving subcategory of $\mod R$.
Then $\X$ contains a nonfree $R$-module $X$, and hence $\syz X\ne0$.
The assumption $\m^2=0$ shows $\m(\syz X)=0$.
As $\syz X$ belongs to $\X$, so does $k$.
Since $R$ is Artinian, every $R$-module has finite length and can be constructed from $k$ by taking extensions.
Hence every $R$-module belongs to $\X$.
\end{proof}

For a Cohen-Macaulay local ring with minimal multiplicity, a result analogous with Lemmas \ref{xkx} and \ref{syzd} holds.
For an $R$-module $X$, we denote by $\res_RX$ the smallest resolving subcategory of $\mod R$ containing $X$.

\begin{prop}\label{sdxr}
Let $R$ be a Cohen-Macaulay local ring with minimal multiplicity.
Suppose that the residue field $k$ is infinite.
Let $\X$ be a resolving subcategory of $\mod R$ contained in $\CM(R)$ containing a nonfree $R$-module.
Then $\syz^dk$ belongs to $\X$.
\end{prop}

\begin{proof}
Since $k$ is infinite, there exists a parameter ideal $Q=(x_1,\dots,x_d)$ of $R$ such that $\m^2=Q\m$; see \cite[Exercise 4.6.14]{BH}.
Hence $(\m/Q)^2=0$ holds in $R/Q$.
Let $X\in\X$ be a nonfree $R$-module.
Since $X$ is Cohen-Macaulay and $Q$ is generated by a regular sequence, $X/QX$ is a nonfree $R/Q$-module (cf. \cite[Lemma 1.3.5]{BH}).
Applying Lemma \ref{m^2} to the local ring $R/Q$, we see that $\res_{R/Q}X/QX$ coincides with $\mod R/Q$, hence it contains $k$.
Making use of \cite[Lemma 5.8]{stcm} repeatedly ($d$ times), we observe that $\syz_R\syz_{R/(x_1)}\cdots\syz_{R/(x_1,\dots,x_{d-1})}k$ is in $\res_RX$, hence in $\X$.
It is easily checked that this is isomorphic to $\syz_R^dk$ up to $R$-free summand, and thus $\syz_R^dk$ belongs to $\X$.
\end{proof}

\begin{rem}
The last part of the above proof may be replaced with the following:

Since $k$ is in $\res_{R/Q}X/QX$, we see by \cite[Remark 3.2(4)]{res} that $k$ is also in $\res_R(R/Q\oplus X/QX)$.
Hence it is observed by \cite[Remark 3.2(4)]{res} again that $\syz_R^dk$ belongs to $\res_R\syz_R^d(R/Q\oplus X/QX)=\res_R\syz_R^d(X/QX)$.
Applying \cite[Lemma 4.3]{stcm} to the Koszul complex of $x_1,\dots,x_d$ with coefficients in $X$, we see that $\syz_R^d(X/QX)$ is in $\res_RX$.
Therefore $\syz_R^dk$ belongs to $\res_RX$, hence to $\X$.
\end{rem}

Here we check that the following elementary lemma holds.

\begin{lem}\label{inres}
The field $\kappa(\p)$ is infinite for any nonmaximal prime ideal $\p$ of $R$.
\end{lem}

\begin{proof}
As $R/\p$ is an integral domain that is not a field, it is an infinite set.
Since $\kappa(\p)$ is the quotient field of $R/\p$, we have that $\kappa(\p)$ contains $R/\p$.
Hence $\kappa(\p)$ is also infinite.
\end{proof}

Now we can prove the following proposition, which is the essential part of the main result of this section.

\begin{prop}\label{mmpr}
Let $R$ be a Cohen-Macaulay local ring locally with minimal multiplicity on the punctured spectrum.
Let $\X$ be a resolving subcategory of $\mod R$ contained in $\CM(R)$ containing $\syz^dk$.
If $M$ is a Cohen-Macaulay $R$-module such that $\V(M)$ is contained in $\V(\X)$, then $M$ belongs to $\X$.
\end{prop}

\begin{proof}
We induce on $n:=\dim\V_R(M)$.
If $n=-\infty$ or $n=0$, then the conclusion follows from Lemma \ref{cv}(1)(2).
Let $n>0$.
Fix $\p\in\min\V_R(M)$.
Then $\p\ne\m$.
By assumption, $R_\p$ has minimal multiplicity.
Since $\p\in\V_R(\X)$, a nonfree $R_\p$-module exists in $\add_{R_\p}\X_\p$.
Hence $\syz^{\dim R_\p}\kappa(\p)\in\add_{R_\p}\X_\p$ by Lemmas \ref{cv}(4), \ref{inres} and Proposition \ref{sdxr}.
We have $\V_{R_\p}(M_\p)=\{\p R_\p\}\subseteq\V_{R_\p}(\add_{R_\p}\X_\p)$ by Lemma \ref{cv}(3).
Hence $\dim\V_{R_\p}(M_\p)=0<n$, and the induction hypothesis implies $M_\p\in\add_{R_\p}\X_\p$.
By Lemma \ref{cv}(5), there is an exact sequence $0 \to L \to N \to X \to 0$ of Cohen-Macaulay $R$-modules such that $X$ belongs to $\X$, that $M$ is a direct summand of $N$, that $\V_R(L)$ is contained in $\V_R(M)$ and that the inequality $\dim\V_R(L)<\dim\V_R(M)$ holds.
The induction hypothesis shows $L\in\X$.
The above exact sequence implies $M\in\X$.
\end{proof}

The theorem below is the main result of this section, which yields a classification of resolving subcategories of Cohen-Macaulay modules.

\begin{thm}\label{mmthm}
Let $R$ be a Cohen-Macaulay singular local ring which locally has minimal multiplicity on the punctured spectrum.
One has a one-to-one correspondence:
$$
\begin{CD}
\left\{
\begin{matrix}
\text{Resolving subcategories}\\
\text{of $\mod R$ contained in $\CM(R)$}\\
\text{and containing $\syz^dk$}
\end{matrix}
\right\}
\begin{matrix}
@>{\V}>>\\
@<<{\V^{-1}_{\CM}}<
\end{matrix}
\left\{
\begin{matrix}
\text{Nonempty specialization-closed}\\
\text{subsets of $\Spec R$}\\
\text{contained in $\Sing R$}
\end{matrix}
\right\}.
\end{CD}
$$
\end{thm}

\begin{proof}
Let $\X$ be a resolving subcategory of $\mod R$ containing $\syz^dk$.
Since $R$ is singular, we have $\V(\X)\ne\emptyset$.
If $\Phi\ne\emptyset$ is a specialization-closed subset of $\Spec R$, then we have $\m\in\Phi$, and $\syz^dk\in\V^{-1}_{\CM}(\Phi)$ since $\V(\syz^dk)=\{\m\}$.
Lemma \ref{cv}(6)(8) guarantees that $\V,\V^{-1}_{\CM}$ are well-defined maps.
Lemma \ref{cv}(9)(10) and Proposition \ref{mmpr} show that the two maps form mutually inverse bijections.
\end{proof}

\begin{rem}
Over a hypersurface $R$, a classification theorem of all the resolving subcategories of $\mod R$ contained in $\CM(R)$ is obtained in \cite[Main Theorem]{stcm}.
The proof of this result heavily relies on \cite[Proposition 5.9]{stcm}, whose proof uses the fact that the localization of a hypersurface at a prime ideal is again a hypersurface.
The localization at a prime ideal of a Cohen-Macaulay local ring with minimal multiplicity does not necessarily have minimal multiplicity.
In fact, with the notation of Example \ref{mmlclz} below, the local ring $R$ has minimal multiplicity, but the localization $R_\p$ does not.
Thus, the same method does not yield a classification theorem of all the resolving subcategories of $\mod R$ contained in $\CM(R)$ over a Cohen-Macaulay local ring $R$ with minimal multiplicity.
\end{rem}

\section{Finite Cohen-Macaulay representation type}\label{sec6}

In this section, we consider the classification problem of resolving subcategories over a Cohen-Macaulay local ring of finite Cohen-Macaulay representation type.
Recall that a Cohen-Macaulay local ring $R$ has {\em finite Cohen-Macaulay representation type} if there are only finitely many nonisomorphic indecomposable Cohen-Macaulay $R$-modules.
First of all, we define the left and right perpendicular subcategories of a given subcategory, and introduce the notions of a right approximation and a contravariantly finite subcategory.

\begin{dfn}
\begin{enumerate}[(1)]
\item
Let $\X$ be a subcategory of $\mod R$.
We denote by ${}^\perp\X$ (respectively, $\X^\perp$) the subcategory of $\mod R$ consisting of all $R$-modules $M$ satisfying $\Ext_R^i(M,X)=0$ (respectively, $\Ext_R^i(X,M)=0$) for all $X\in\X$ and all $i>0$.
\item
Let $R$ be a Cohen-Macaulay local ring.
Let $\X$ be a subcategory of $\CM(R)$.
We denote by $\lp\X$ (respectively, $\X\rp$) the subcategory of $\CM(R)$ consisting of all Cohen-Macaulay $R$-modules $M$ satisfying $\Ext_R^i(M,X)=0$ (respectively, $\Ext_R^i(X,M)=0$) for all $X\in\X$ and all $i>0$.
\end{enumerate}
\end{dfn}

\begin{dfn}
Let $\X$ be a subcategory of $\mod R$.
Let $\phi:X\to M$ be a homomorphism of $R$-modules with $X\in\X$.
We say that $\phi$ is a {\em right $\X$-approximation} (of $M$) if the homomorphism $\Hom_R(X',\phi):\Hom_R(X',X)\to\Hom_R(X',M)$ is surjective for all $X'\in\X$, in other words, every homomorphism $X'\to M$ with $X'\in\X$ factors through $\phi$.
We say that $\X$ is {\em contravariantly finite} if all modules $M\in\mod R$ admit right $\X$-approximations.
\end{dfn}

We make here several statements for later use.

\begin{lem}\label{compos}
\begin{enumerate}[\rm (1)]
\item
For each subcategory $\X$ of $\mod R$, one has that ${}^\perp\X$ is a resolving subcategory of $\mod R$.
\item
Let $\X$ and $\Y$ be subcategories of $\mod R$, and assume that $\Y$ is contained in $\X$.
Let $Y \overset{g}{\to} X \overset{f}{\to} M$ be homomorphisms of $R$-modules.
If $f$ is a right $\X$-approximation and $g$ is a right $\Y$-approximation, then $fg$ is a right $\Y$-approximation.
\item
Let $R$ be a Henselian local ring, and let $\X$ be a resolving subcategory of $\mod R$.
Suppose that $M$ admits a right $\X$-approximation.
Then $M$ admits a surjective right $\X$-approximation whose kernel belongs to $\X^\perp$.
Hence there exists an exact sequence $0 \to Y \to X \to M \to 0$ of $R$-modules with $X\in\X$ and $Y\in\X^\perp$.
\item
Let $R$ be a Cohen-Macaulay local ring with a canonical module.
Then $\CM(R)$ is a contravariantly finite resolving subcategory of $\mod R$.
\end{enumerate}
\end{lem}

\begin{proof}
(1) and (2) are straightforward.
(3) is essentially proved in \cite[Proposition 3.3]{AR}.
The proof is explicitly stated in \cite[Lemma 3.8]{arg}.
(4) follows from \cite[Examples 2.5(3) and 3.3(5)]{arg}.
(The assumption that $R$ is Henselian is not used there.)
\end{proof}

Let $R$ be a Cohen-Macaulay local ring with a canonical module $\omega$.
Let $M$ be a Cohen-Macaulay $R$-module and $\X$ a subcategory of $\CM(R)$.
Then we set $M^\dag=\Hom_R(M,\omega)$, and denote by $\X^\dag$ the subcategory of $\mod R$ consisting of all modules of the form $X^\dag$ with $X\in\X$.
Note that $M^\dag$ is also a Cohen-Macaulay $R$-module, and hence $\X^\dag$ is contained in $\CM(R)$.
Note also that $M^\dag$ is uniquely determined up to isomorphism since so is a canonical module, and hence $\X^\dag$ is uniquely determined.

\begin{lem}\label{perpd}
Let $R$ be a Cohen-Macaulay local ring with a canonical module $\omega$.
Let $\X$ be a subcategory of $\CM(R)$.
Let $M$ be a module in $\X\rp$.
Then $M^\dag$ belongs to $\lp(\X^\dag)$.
If $\X$ contains $\omega$, then $\Ext_R^i(M^\dag,R)=0$ for every $i>0$.
\end{lem}

\begin{proof}
Let $X$ be any $R$-module in $\X$.
Then $X$ is Cohen-Macaulay, and we have
\begin{align*}
\RHom_R(M^\dag,X^\dag) & \cong \RHom_R(M^\dag,\RHom_R(X,\omega))
\cong \RHom_R(X,\RHom_R(M^\dag,\omega)) \\
& \cong \RHom_R(X,(M^\dag)^\dag) \cong \RHom_R(X,M)
\end{align*}
in the derived category of $\mod R$.
Since $M$ is in $\X\rp$, we obtain $\Ext_R^i(M^\dag,X^\dag)\cong\Ext_R^i(X,M)=0$ for every $i>0$.
Thus $M^\dag$ belongs to $\lp(\X^\dag)$.

If $\omega\in\X$, then $\X^\dag$ contains $\omega^\dag\cong R$.
We have $\Ext_R^i(M^\dag,R)=0$ for every $i>0$.
\end{proof}

Now we can prove a result on determination of contravariantly finite resolving subcategories over a Cohen-Macaulay Henselian local ring.

\begin{thm}\label{lp}
Let $R$ be a Cohen-Macaulay Henselian local ring with a canonical module $\omega$.
Let $\X$ be a contravariantly finite resolving subcategory of $\mod R$ contained in $\CM(R)$ containing $\omega$.
Then $\X$ coincides with either $\add R$ or $\CM(R)$.
\end{thm}

\begin{proof}
(1) First, we prove that $\lp(\X^\dag)$ is a contravariantly finite resolving subcategory of $\mod R$.
The resolving property of $\lp(\X^\dag)$ is easily verified by using Lemma \ref{compos}(1) and \cite[Example 1.6(3)]{stcm}.
By Lemma \ref{compos}(2)(4), it is enough to show that each $M\in\CM(R)$ admits a right $\lp(\X^\dag)$-approximation.
Lemma \ref{compos}(3) gives an exact sequence
$$
0 \to Y \to X \to (\syz M)^\dag \to 0
$$
with $X\in\X$ and $Y\in\X^\perp$.
As $X$ and $(\syz M)^\dag$ are Cohen-Macaulay, so is $Y$, and hence $Y\in\X\rp$.
Lemma \ref{perpd} implies $Y^\dag\in\lp(\X^\dag)$.
Dualizing the above exact sequence by $\omega$, we get an exact sequence $0 \to \syz M \to X^\dag \to Y^\dag \to 0$.
Also, there is an exact sequence $0 \to \syz M \to R^{\oplus n} \to M \to 0$.
We make the following pushout diagram.
$$
\begin{CD}
@. 0 @. 0 \\
@. @VVV @VVV \\
0 @>>> \syz M @>>> X^\dag @>>> Y^\dag @>>> 0 \\
@. @VVV @VVV @| \\
0 @>>> R^{\oplus n} @>>> Z @>>> Y^\dag @>>> 0 \\
@. @VVV @VVV \\
@. M @= M \\
@. @VVV @VVV \\
@. 0 @. 0
\end{CD}
$$
Lemma \ref{perpd} gives $\Ext_R^i(Y^\dag,R)=0$ for all $i>0$.
Hence the middle row in the above diagram splits, and $Z$ is isomorphic to $Y^\dag\oplus R^{\oplus n}\in\lp(\X^\dag)$.
We get an exact sequence
$$
0 \to X^\dag \to Y^\dag\oplus R^{\oplus n} \overset{\phi}{\to} M \to 0.
$$
Let $W$ be any $R$-module in $\lp(\X^\dag)$.
Then there is an exact sequence
$$
\begin{CD}
\Hom_R(W,Y^\dag\oplus R^{\oplus n}) @>{\Hom_R(W,\phi)}>> \Hom_R(W,M) \to \Ext_R^1(W,X^\dag).
\end{CD}
$$
Since $X^\dag\in\X^\dag$ and $W\in\lp(\X^\dag)$, we have $\Ext_R^1(W,X^\dag)=0$, and the map $\Hom_R(W,\phi)$ is surjective.
Therefore the homomorphism $\phi$ is a right $\lp(\X^\dag)$-approximation.
Consequently, $\lp(\X^\dag)$ is a contravariantly finite resolving subcategory of $\mod R$.

(2) Now, let us prove that either $\X=\add R$ or $\X=\CM(R)$ holds true.
Assume that $\X$ is different from $\add R$.
We want to show that $\X$ coincides with $\CM(R)$.
Let $M$ be a Cohen-Macaulay $R$-module.
Then by Lemma \ref{compos}(3) there exists an exact sequence
\begin{equation}\label{yxm}
0 \to Y \to X \to M \to 0
\end{equation}
of $R$-modules with $X\in\X$ and $Y\in\X^\perp$.
As $X$ and $M$ are Cohen-Macaulay, so is $Y$, and we have $Y\in\X\rp$.
Lemma \ref{perpd} implies that $Y^\dag$ belongs to $\lp(\X^\dag)$ and that $\Ext_R^i(Y^\dag,R)=0$ for every $i>0$.
Thus, by \cite[Theorem 1.4]{arg}, one of the following two statements holds.
\begin{enumerate}[(i)]
\item
As a subcategory of $\mod R$, $\lp(\X^\dag)$ coincides with either $\mod R$ or $\CM(R)$.
\item
The $R$-module $Y^\dag$ has finite projective dimension.
\end{enumerate}

Suppose that the statement (i) holds.
Then $\syz^dk$ belongs to $\lp(\X^\dag)$.
Let $Z$ be a module in $\X$.
Then we have $0=\Ext_R^i(\syz^dk,Z^\dag)\cong\Ext_R^{i+d}(k,Z^\dag)$ for all $i>0$, which implies that the Cohen-Macaulay $R$-module $Z^\dag$ has finite injective dimension.
Hence $Z^\dag$ is isomorphic to a direct sum of copies of $\omega$, and therefore $Z$ is free.
Thus we have $\X=\add R$, which contradicts our assumption.

Consequently, the statement (i) cannot hold true, and the statement (ii) must hold.
Since $Y^\dag$ is a Cohen-Macaulay $R$-module, it is free.
Hence $Y$ is isomorphic to a direct sum of copies of $\omega$.
Then the exact sequence \eqref{yxm} splits, and $M$ is isomorphic to a direct summand of $X$.
As $\X$ is closed under direct summands, $M$ belongs to $\X$.
Now we conclude that $\X$ coincides with $\CM(R)$.
\end{proof}

Next we want to exclude the Henselian assumption on $R$ from Theorem \ref{lp}.
For this, we need to lift the contravariant finite and resolving properties of a subcategory of $\mod R$ to the completion of $R$.
We consider this in the following two results.

Let $R\to S$ be a flat homomorphism of local rings.
For a subcategory $\X$ of $\mod R$, we denote by $\X\otimes_RS$ the subcategory of $\mod S$ consisting of all $S$-modules of the form $X\otimes_RS$ with $X\in\X$.
Note that if $\X$ is closed under direct sums, then so is $\X\otimes_RS$.

\begin{lem}\label{nagai}
\begin{enumerate}[\rm (1)]
\item
Let $\X$ be a subcategory of $\mod R$.
Let $\phi:X\to M$ be a right $\X$-approximation, and let $\pi:M\to N$ be a split epimorphism of $R$-modules.
Then $\pi\phi:X\to N$ is a right $\X$-approximation.
\item
Let $R\to S$ be a flat (not necessarily local) homomorphism of local rings.
Let $\X$ be a subcategory of $\mod R$ closed under direct sums.
If $\phi:X\to M$ is a right $\X$-approximation, then $\phi\otimes_RS:X\otimes_RS\to M\otimes_RS$ is a right $\add_S(\X\otimes_RS)$-approximation.
\end{enumerate}
\end{lem}

\begin{proof}
(1) There is an $R$-homomorphism $\theta:N\to M$ with $\pi\theta=1$.
Let $f:X'\to N$ be an $R$-homomorphism with $X'\in\X$.
Then $\theta f:X'\to M$ factors through $\phi$, namely, the equality $\theta f=\phi g$ holds for some $R$-homomorphism $g:X'\to X$.
We have $f=\pi\theta f=\pi\phi g$, whence $f$ factors through $\pi\phi$.

(2) Let $Y$ be an $S$-module in $\add_S(\X\otimes S)$, and let $f:Y\to M\otimes S$ be an $S$-homomorphism.
Then there are modules $X'\in\X$ and $Z\in\mod S$ such that $Y\oplus Z$ is isomorphic to $X'\otimes S$.
We have homomorphisms $\theta:Y\to X'\otimes S$ and $\pi:X'\otimes S\to Y$ of $S$-modules with $\pi\theta=1$.
The homomorphism $f\pi$ is in $\Hom_S(X'\otimes_RS,M\otimes_RS)\cong\Hom_R(X',M)\otimes_RS$, and we can write $f\pi=\sum_{j=1}^mg_j\otimes s_j$ for some $g_j\in\Hom_R(X',M)$ and $s_j\in S$.
Since $\phi$ is a right $\X$-approximation, there is an $R$-homomorphism $h_j:X'\to X$ such that $g_j=\phi h_j$ for $1\le j\le m$.
We have equalities
$$
f=f\pi\theta=\left(\sum_{j=1}^m(\phi h_j)\otimes s_j\right)\theta=(\phi\otimes S)\left(\sum_{j=1}^mh_j\otimes s_j\right)\theta.
$$
Hence $f$ factors through $\phi\otimes S$.
\end{proof}

Let $\widehat R$ be the completion of $R$ in the $\m$-adic topology.
For a subcategory $\X$ of $\mod R$, we denote by $\widehat\X$ the subcategory of $\mod\widehat R$ consisting of all $\widehat R$-modules of the form $\widehat X$ with $X\in\X$.
This can be identified as $\X\otimes_R\widehat R$.
When $R$ is a Cohen-Macaulay local ring, we denote by $\CM_0(R)$ the subcategory of $\CM(R)$ consisting of all Cohen-Macaulay $R$-modules that are locally free on the punctured spectrum of $R$.

\begin{prop}\label{asda}
\begin{enumerate}[\rm (1)]
\item
Let $\X$ be a resolving subcategory of $\mod R$.
Assume that all $R$-modules in $\X$ are locally free on the punctured spectrum of $R$.
Then $\add_{\widehat R}\widehat\X$ is a resolving subcategory of $\mod\widehat R$.
\item
Let $R$ be a Cohen-Macaulay local ring, and let $\X$ be a resolving subcategory of $\mod R$ contained in $\CM_0(R)$.
Then the following hold.
\begin{enumerate}[\rm (i)]
\item
The subcategory $\add_{\widehat R}\widehat\X$ of $\mod\widehat R$ is contained in $\CM_0(\widehat R)$.
\item
If every module in $\CM_0(R)$ admits a right $\X$-approximation, then every module in $\CM_0(\widehat R)$ admits a right $\add_{\widehat R}\widehat\X$-approximation
\end{enumerate}
\end{enumerate}
\end{prop}

\begin{proof}
(1) Clearly, $\add_{\widehat R}\widehat\X$ is closed under direct summands and contains $\widehat R$.
Let $M$ be an $\widehat R$-module in $\add_{\widehat R}\widehat\X$.
Then $M$ is isomorphic to a direct summand of $\widehat X$ for some $X\in\X$.
Hence $\syz_{\widehat R}M$ is isomorphic to a direct summand of $\syz_{\widehat R}\widehat X\cong\widehat{\syz_RX}$.
Since $\syz_RX$ is in $\X$, the module $\syz_{\widehat R}M$ is in $\add_{\widehat R}\widehat\X$.
Therefore $\add_{\widehat R}\widehat\X$ is closed under syzygies.
To check the closedness under extensions, let $0 \to L \to M \to N \to 0$ be an exact sequence of $\widehat R$-modules with $L,N\in\add_{\widehat R}\widehat\X$.
Then we have isomorphisms $L\oplus L'\cong\widehat X$ and $N\oplus N'\cong\widehat Y$ for some $\widehat R$-modules $L',N'$ and some $R$-modules $X,Y\in\X$.
Taking the direct sum of the above exact sequence and trivial exact sequences $0 \to L\overset{=}{\to} L \to 0 \to 0$ and $0 \to 0 \to N \overset{=}{\to} N \to 0$, we have an exact sequence
$$
\sigma: 0 \to \widehat X \to M' \to \widehat Y \to 0
$$
of $\widehat R$-modules, where $M'=M\oplus L'\oplus N'$.
This corresponds to an element of $\Ext_{\widehat R}^1(\widehat Y,\widehat X)$.
Since $Y$ is locally free on the punctured spectrum of $R$, the $R$-module $\Ext_R^1(Y,X)$ has finite length.
Hence it is a complete $R$-module, and the natural map $\Ext_R^1(Y,X)\to\Ext_{\widehat R}^1(\widehat Y,\widehat X)$ is an isomorphism.
This shows that there exists an exact sequence
$$
\tau: 0 \to X \to Z \to Y \to 0
$$
of $R$-modules such that the exact sequence $\widehat\tau: 0 \to \widehat X \to \widehat Z \to \widehat Y \to 0$ of $\widehat R$-modules is equivalent to $\sigma$ as an extension of $\widehat Y$ by $\widehat X$, whence $\widehat Z$ is isomorphic to $M'$.
Since the subcategory $\X$ of $\mod R$ is closed under extensions, the module $Z$ is in $\X$, and so $M$ belongs to $\add_{\widehat R}\widehat\X$.
Consequently, $\add_{\widehat R}\widehat\X$ is closed under extensions as a subcategory of $\mod\widehat R$.
Thus we conclude that $\add_{\widehat R}\widehat\X$ is a resolving subcategory of $\mod\widehat R$.

(2)(i) We have only to show that $\widehat{X}_\q$ is $\widehat{R}_\q$-free for every $X\in\X$ and every nonmaximal prime ideal $\q$ of $\widehat R$.
Putting $\p=\q\cap R$, we observe that $\p$ is a nonmaximal prime ideal of $R$ (cf. \cite[Theorem 15.1]{M}).
Hence $X_\p$ is $R_\p$-free.
Since there are isomorphisms $\widehat{X}_\q\cong X\otimes_R\widehat{R}_\q\cong X_\p\otimes_{R_\p}\widehat{R}_\q$, the module $\widehat{X}_\q$ is $\widehat{R}_\q$-free.

(ii) Let $M$ be an $\widehat R$-module in $\CM_0(\widehat R)$.
It follows from \cite[Corollary 3.3]{stcm} that $M$ is isomorphic to a direct summand of $\widehat N$ for some $N\in\CM_0(R)$.
Hence there exists a split epimorphism $\pi:\widehat N\to M$ of $\widehat R$-modules.
By assumption, there is a right $\X$-approximation $\phi:X\to N$.
Lemma \ref{nagai}(2) implies that the completion $\widehat\phi:\widehat X\to\widehat N$ is a right $\add_{\widehat R}\widehat\X$-approximation.
Then the composite map $\pi\widehat\phi:\widehat X\to M$ is also a right $\add_{\widehat R}\widehat\X$-approximation by Lemma \ref{nagai}(1).
\end{proof}

Now we can prove the following theorem.
We can actually exclude from Theorem \ref{lp} the assumption that $R$ is Henselian, if instead we assume that the completion of $R$ has an isolated singularity.

\begin{thm}\label{addorcm}
Let $R$ be a Cohen-Macaulay local ring with a canonical module $\omega$.
Suppose that $\widehat R$ has an isolated singularity.
Let $\X$ be a contravariantly finite resolving subcategory of $\mod R$ contained in $\CM(R)$ containing $\omega$.
Then $\X$ is either $\add R$ or $\CM(R)$.
\end{thm}

\begin{proof}
Since $\widehat R$ has an isolated singularity, so does $R$ (cf. \cite[Proposition 3.4]{stcm}).
Hence we have $\CM(\widehat R)=\CM_0(\widehat R)$ and $\CM(R)=\CM_0(R)$.
Propositions \ref{asda}(2), \ref{asda}(1) and Lemma \ref{compos}(2)(4) show that $\add_{\widehat R}\widehat\X$ is a contravariantly finite resolving subcategory of $\mod\widehat R$.
Since $\widehat R$ is Henselian, it follows from Theorem \ref{lp} that $\add_{\widehat R}\widehat\X$ coincides with either $\add_{\widehat R}\widehat R$ or $\CM(\widehat R)$.
In the former case we have $\X=\add R$, so let us consider the case where $\add_{\widehat R}\widehat\X=\CM(\widehat R)$ holds.
We want to prove that $\X=\CM(R)$.
According to \cite[Corollary 2.7]{stcm}, it is enough to show that $\syz^dk\in\X$.
By the contravariant finiteness of $\X$, there is a right $\X$-approximation $\phi:X\to\syz^dk$.
There is a surjective homomorphism from a free $R$-module to $\syz^dk$, and it factors through $\phi$.
Hence $\phi$ is surjective, and we have an exact sequence
$$
\sigma: 0 \to L \to X \overset{\phi}{\to} \syz^dk \to 0
$$
of $R$-modules.
Taking the completion gives an exact sequence
$$
\widehat\sigma: 0 \to \widehat L \to \widehat X \overset{\widehat\phi}{\to} \widehat{\syz^dk} \to 0.
$$
Lemma \ref{nagai}(2) says that $\widehat\phi$ is a right $\add_{\widehat R}\widehat\X$-approximation.
Note that $\widehat{\syz^dk}$ is in $\CM(\widehat R)=\add_{\widehat R}\widehat\X$.
Hence the identity map $\id_{\widehat{\syz^dk}}:\widehat{\syz^dk}\to\widehat{\syz^dk}$ factors through $\widehat\phi$, which implies that $\widehat\phi$ is a split epimorphism.
We have $\widehat\sigma=0$ in $\Ext_{\widehat R}^1(\widehat{\syz^dk},\widehat L)$.
As the $R$-module $\Ext_R^1(\syz^dk,L)$ has finite length, the natural map $\Ext_R^1(\syz^dk,L)\to\Ext_{\widehat R}^1(\widehat{\syz^dk},\widehat L)$ is an isomorphism.
Hence $\sigma=0$ in $\Ext_R^1(\syz^dk,L)$.
Thus $\syz^dk$ is isomorphic to a direct summand of $X\in\X$, and $\syz^dk$ belongs to $\X$.
\end{proof}

\begin{cor}\label{wasur}
Let $R$ be a Cohen-Macaulay local ring admitting a canonical module $\omega$.
Assume that $R$ is of finite Cohen-Macaulay representation type and that the completion $\widehat R$ has an isolated singularity.
Let $\X$ be a resolving subcategory of $\mod R$ contained in $\CM(R)$ containing $\omega$.
Then one has either $\X=\add R$ or $\X=\CM(R)$.
\end{cor}

\begin{proof}
Since $R$ has finite Cohen-Macaulay representation type, there exist only a finite number of nonisomorphic indecomposable Cohen-Macaulay $R$-modules.
Let $C_1,\dots,C_n$ be a complete list of them.
We may assume that $C_i$ belongs to $\X$ for $1\le i\le m$ and does not belong to $\X$ for $m+1\le i\le n$.
Then we easily see that $\X=\add(\bigoplus_{i=1}^mC_i)$ holds.
It follows from (the proof of) \cite[Proposition 4.2]{AS} that $\X$ is contravariantly finite.
Now the conclusion follows from Theorem \ref{addorcm}.
\end{proof}

The rest of this section is devoted to the main problem of this paper, that is, the classification problem of resolving subcategories.
Taking advantage of the above corollary, we prove the following proposition, which is the essential part of the classification theorem.

\begin{prop}\label{fcmpr}
Let $R$ be a Cohen-Macaulay local ring with a canonical module $\omega$.
Suppose that for every nonmaximal prime ideal $\p$ the local ring $R_\p$ has finite Cohen-Macaulay representation type whose completion has an isolated singularity.
Let $\X$ be a resolving subcategory of $\mod R$ contained in $\CM(R)$ containing $\omega$ and $\syz^dk$.
Let $M$ be a Cohen-Macaulay $R$-module.
If $\V(M)$ is contained in $\V(\X)$, then $M$ belongs to $\X$.
\end{prop}

\begin{proof}
The proposition is proved by induction on $n:=\dim\V_R(M)$.
Lemma \ref{cv}(1)(2) handles the cases $n=-\infty,0$.
When $n\ge 1$, let $\p\in\min\V(M)$.
Then $\p\ne\m$, whence $R_\p$ has finite Cohen-Macaulay representation type and its completion has an isolated singularity.
Lemma \ref{cv}(4) and Corollary \ref{wasur} imply that $\add_{R_\p}\X_\p=\CM(R_\p)$, which contains $M_\p$.
Lemma \ref{cv}(5) gives an exact sequence $0 \to L \to N \to X \to 0$ of Cohen-Macaulay $R$-modules with $X\in\X$, $\V(L)\subseteq\V(M)$ and $\dim\V(L)<n$ such that $M$ is a direct summand of $N$.
The induction hypothesis shows $L\in\X$, and so $M\in\X$.
\end{proof}

Now we obtain a classification theorem of resolving subcategories.

\begin{thm}\label{fcmt}
Let $R$ be a Cohen-Macaulay singular local ring with a canonical module $\omega$.
Suppose that for each nonmaximal prime ideal $\p$ the local ring $R_\p$ has finite Cohen-Macaulay representation type whose completion has an isolated singularity.
Then one has the following one-to-one correspondence:
$$
\begin{CD}
\left\{
\begin{matrix}
\text{Resolving subcategories}\\
\text{of $\mod R$ contained in $\CM(R)$}\\
\text{and containing $\omega,\syz^dk$}
\end{matrix}
\right\}
\begin{matrix}
@>{\V}>>\\
@<<{\V^{-1}_{\CM}}<
\end{matrix}
\left\{
\begin{matrix}
\text{Nonempty specialization-closed}\\
\text{subsets of $\Spec R$ contained in $\Sing R$}\\
\text{and containing $\ng(R)$}
\end{matrix}
\right\}.
\end{CD}
$$
\end{thm}

\begin{proof}
We have $\ng(R)=\V(\omega)$.
The assertion is proved by the same argument as the proof of Theorem \ref{mmthm} except that Proposition \ref{mmpr} is replaced with Proposition \ref{fcmpr}.
\end{proof}

\begin{cor}\label{fcrtcor}
Let $R$ be a Cohen-Macaulay excellent singular local ring with a canonical module $\omega$ locally of finite Cohen-Macaulay representation type on the punctured spectrum.
Then one has the same one-to-one correspondence as in Theorem \ref{fcmt}.
\end{cor}

\begin{proof}
Let $\p$ be a nonmaximal prime ideal of $R$.
Since $R$ is excellent, so is $R_\p$ (cf. \cite[Page 260]{M}).
The assumption that $R_\p$ is of finite Cohen-Macaulay representation type and \cite[Corollary 2]{HL} imply that $R_\p$ has an isolated singularity, hence so does the completion of the local ring $R_\p$ (cf. \cite[Proposition 3.4]{stcm}).
\end{proof}

\section{Examples}\label{sec7}

Combining the results obtained in the previous two sections with \cite[Theorem 5.13]{stcm}, now we have classifications of resolving subcategories over a Cohen-Macaulay local ring $R$ in each of the following three cases:
\begin{enumerate}[(a)]
\item
$R$ locally has minimal multiplicity on the punctured spectrum;
\item
$R$ locally has finite Cohen-Macaulay representation type on the punctured spectrum;
\item
$R$ is locally a hypersurface on the punctured spectrum.
\end{enumerate}

In this section, we construct several examples of a Cohen-Macaulay complete singular local ring $R$ satisfying some of these three conditions.
If $R$ has an isolated singularity, then we have $\Sing R=\{\m\}$, and the specialization-closed subsets of $\Spec R$ contained in $\Sing R$ are trivial.
Thus we shall construct examples so that $R$ does not have an isolated singularity.
Throughout the examples below, let $k$ be a field.

\begin{ex}
Let $R=k[[x,y,z]]/(x^2,xz,yz)$.

(1) The ring $R$ is a $1$-dimensional Cohen-Macaulay non-Gorenstein complete local ring with a parameter $y-z$.
The prime ideals of $R$ are $\p=(x,y)$, $\q=(x,z)$ and $\m=(x,y,z)$.
Since $R$ is not reduced, $R$ does not have an isolated singularity.

(2) We have $R_\p\cong k((z))$ and $R_\q\cong k[[x,y]]_{(x)}/(x^2)$, showing that $R$ is locally a hypersurface of finite Cohen-Macaulay representation type with minimal multiplicity on the punctured spectrum.
\end{ex}

\begin{ex}
Let $R=k[[x,y,z]]/(x^2,xy,y^2)$.

(1) The ring $R$ is a $1$-dimensional Cohen-Macaulay non-Gorenstein complete local ring.
The element $z$ is a parameter of $R$, and all the prime ideals are $\p=(x,y)$ and $\m=(x,y,z)$.
As $R$ is not reduced, $R$ does not have an isolated singularity.

(2) Since $(\p R_\p)^2=0$, $R$ is locally with minimal multiplicity on the punctured spectrum.

(3) The local ring $R_\p$ is not a hypersurface.
Hence $R$ is neither locally a hypersurface nor locally has finite Cohen-Macaulay representation type on the punctured spectrum.
\end{ex}

The following example is due to Shiro Goto.

\begin{ex}\label{goto}
Let $R=k[[x,y,z,w]]/(xy,z^2,zw,w^2)$.

(1) The ring $R$ is a Cohen-Macaulay non-Gorenstein complete local ring of dimension one that does not have minimal multiplicity.
All the prime ideals are $\p=(x,z,w)$, $\q=(y,z,w)$ and $\m=(x,y,z,w)$.
The element $x-y$ is a parameter of $R$.
Since it is not reduced, $R$ is not with an isolated singularity.

(2) The ring $R$ is neither locally a hypersurface nor locally of finite Cohen-Macaulay representation type on the punctured spectrum.
Indeed, $R_\p$ is isomorphic to the Artinian local ring $k[[y,z,w]]_{(z,w)}/(z^2,zw,w^2)$, which is not a hypersurface.

(3) We have $(\p R_\p)^2=0$ and $(\q R_\q)^2=0$.
Therefore $R$ locally has minimal multiplicity on the punctured spectrum.
\end{ex}

\begin{ex}\label{mmlclz}
Let $R=k[[x,y,z]]/(x^2-yz,xy,y^2)$.

(1) The ring $R$ is a $1$-dimensional Cohen-Macaulay non-Gorenstein complete local ring with a parameter $z$.
The prime ideals of $R$ are $\p=(x,y)$ and $\m=(x,y,z)$.

(2) Set $S=k[[x,y,z]]_{(x,y)}/(y-\frac{x^2}{z})$.
Then $S$ is a discrete valuation ring with maximal ideal $xS$.
The local ring $R_\p=S/(x^3)$ is an Artinian hypersurface and hence $R_\p$ has finite Cohen-Macaulay representation type.
But $R_\p$ does not have minimal multiplicity since $(\p R_\p)^2\ne 0$.
Hence $R$ is locally a hypersurface of finite Cohen-Macaulay representation type on the punctured spectrum, but $R$ does not locally have minimal multiplicity.
\end{ex}

\begin{ex}
Suppose that $k$ has characteristic zero.
Let $R=k[[x,y,z,w]]/(xy,xz,yz)$.

(1) The ring $R$ is a $2$-dimensional, Cohen-Macaulay, non-Gorenstein, reduced, complete local ring.
The elements $x+y+z,w$ form a system of parameters of $R$.

(2) Since $R$ is not a (normal) domain, $R$ does not have an isolated singularity.

(3) The ring $R$ is locally of finite Cohen-Macaulay representation type and locally with minimal multiplicity on the punctured spectrum.
Indeed, let $\p$ be a nonmaximal prime ideal of $R$.
Then at least one of $x,y,z,w$ is not in $\p$.
If $x\notin\p$, then $x$ is a unit as an element of $R_\p$.
The local ring $R_\p$ is regular, so it has both finite Cohen-Macaulay representation type and minimal multiplicity.
By symmetry, we have the same conclusion when either $y$ or $z$ is not in $\p$.
Now suppose that all of $x,y,z$ are in $\p$.
Then $\p$ contains the ideal $(x,y,z)$ of height two, so we have $\p=(x,y,z)$.
In this case, $w$ is a unit of $R_\p$.
Let $S$ be the completion of the local ring $k[[x,y,z,w]]_{(x,y,z)}$.
The ring $S$ is a $3$-dimensional regular complete local ring containing a field, and $S$ admits a coefficient field $K$.
Applying Cohen's structure theorem, we observe that the completion of the local ring $R_\p$ is isomorphic to the ring $T:=K[[X,Y,Z]]/(XY,XZ,YZ)$.
This is a $1$-dimensional Cohen-Macaulay local ring with a parameter $X+Y+Z$.
We see that $T$ has minimal multiplicity, and so does $R_\p$.
Consider the base change $\overline{T}=\overline{K}[[X,Y,Z]]/(XY,XZ,YZ)$ of $T$, where $\overline{K}$ denotes the algebraic closure of the field $K$.
Then $\overline{T}$ has finite Cohen-Macaulay representation type by \cite[Remark (9.16)]{Y}.
Since $\overline{T}$ is faithfully flat over $R_\p$, the ring $R_\p$ also has finite Cohen-Macaulay representation type by \cite[Theorem 1.5]{W}.

(4) The ring $R$ is not locally a hypersurface on the punctured spectrum.
In fact, for the prime ideal $\p=(x,y,z)$ of $R$, the local ring $R_\p$ is not a hypersurface.
\end{ex}

\begin{rem}
According to \cite[(7.1)]{S}, all known examples of a Cohen-Macaulay complete local $\C$-algebra of finite Cohen-Macaulay representation type that is not a hypersurface have minimal multiplicity.
\end{rem}

\begin{ac}
The author is indebted to the referee for a lot of valuable comments and suggestions.
The author also thanks Shiro Goto very much for constructing Example \ref{goto}.
\end{ac}

\end{document}